\documentclass[a4paper,10pt]{amsart}

\usepackage{extensions}
\begin{document}
\title[]{Strong approximation and central limit theorems for multiscale stochastic gene networks}

\author[B. Huguet]{Baptiste Huguet} 
\address{Univ Rennes, CNRS, IRMAR - UMR 6625, F-35000 Rennes, France}
\email{baptiste.huguet@math.cnrs.fr}
\urladdr{https://www.math.cnrs.fr/~bhuguet/}
\begin{abstract}
We study a mutliscale jump process introduced in a work by Crudu, Debussche, Muller and Radulescu. Using an adequate coupling, we are able to prove the strong convergence, for the uniform topology, to a piecewise deterministic Markov process. Under some additional regularity, we also obtain a central limit theorem and prove that the fluctuations of the continuous scale converge, in a weaker sense, to the solution of a stochastic differential equation. 
\end{abstract}
\maketitle
\setcounter{tocdepth}{1}
\tableofcontents

\section{Introduction}

Molecular biology deals with complex systems of reactions, involving numerous reagents, such as the gene regulatory network. This network models the expression of genes and the synthesis of proteins. Its reagents are mostly mRNA and proteins. Each reagents may have control (inhibition or stimulation) over every reactions. 

A relevant mathematical model of this network should capture diverse features that are experimentally observable, and should be implemented effectively. Stochastic models, and more precisely, Markovian models, have been used to study gene networks since the seminal work of Delbr\"uck \cite{Delbr}, and recent experiments show that these models are more appropriate to study gene networks than deterministic ones (\cite{Kepler}). These models have proved their relevance by reproducing experimentally observed behaviours such as burst-like production \cite{Cai}, emergence of phenotypically distinct subgroups in an isogenic population \cite{kaern}, noise propagation \cite{Paul}, or metastability \cite{Gard}. From a computational perspective, Markovian models are really challenging. Exact algorithms such as the Stochastic Simulation Algorithm (SSA) from \cite{Gill} are extremely time-consuming, especially for highly interconnected systems. 

The search of higher efficiency leads to hybrid, or multiscale, models. The reagents are classified according to their abundances : some species, in large quantity, are treated as continuous variables (as a concentration limit), while the other species, in small amount are still treated as discrete random variables. It results a hybrid process, $\R_+^n\times\N^d$-valued. The work \cite{CDR} suggests that hybrid piecewise deterministic Markov processes (PDMP) are a relevant model for gene regulatory network. These processes have been introduced by \cite{Davis84}, in connection with queuing theory.  Moreover, \cite{CDMR} shows that hybrid PDMP can be obtained as limits, in distribution, of Markovian discrete models, in several situations. The goal of this article is to prove that, using a well-chosen probabilistic representation of the reactions, these convergence results hold for stronger probabilistic convergence. In addition, we sharpen the convergence, with an explicit speed of convergence and a central limit theorem for the fluctuations around the limit. Besides, our proof relies on  simpler arguments than \cite{CDMR}.

The convergence of the mono-scale model has been studied by Kurtz. In \cite{Kurtz70}, he proves the convergence in probability for a pure jump Markov process, $X^N$, with general transition kernel, to the solution, $X$, of an ordinary differential equation (ODE). In \cite{Kurtz}, he also studies the fluctuations, through a central limit theorem, and the diffusion approximation, together with speed of convergence. He uses a coupling representation of $X^N$ and $X$, using independent Poisson processes $P_r$  
\begin{equation*}
\left\{\begin{aligned}
X^N(t) = & X^N_0 + \sum_{r\in\RR}\frac{h_r}{N}P_r\left(N\int_0^t\lambda_r(X^N(s))\, ds\right)\\
X(t) = & X_0 + \sum_{r\in\RR} h_r\int_0^t\lambda_r(X(s))\, ds\\
\end{aligned}\right.
\end{equation*}
where $h_r\in\Z^n$ are the directions of the jumps and $\lambda_r$ the rates. This coupling has been used in \cite{CDMR}, for hybrid processes, but it is not really adequate to the multiscale setting. This explains why they could not obtain a strong convergence result, nor a central limit theorem. In our work, we use a coupling through Poisson random measures. This coupling has been used in several approaches arising from mathematical modelling, especially in epidemic models. The article \cite{PS} studies a population (the abundant scale) in random environment (the jumping discrete scale). They obtain the convergence in probability to a PDMP and results on the extinction time of the epidemic. Their model only allows elementary jumps ($\pm 1/N$ on only one coordinate). In a different approach, \cite{HPV} and then \cite{HPV2} study a spatial epidemic model, with a mean-field point of view. Their hybrid process consists in the spatial position of an individual (the continuous scale) and its infectious state (discrete scale). In this case, the limit of the position jump process is not interpreted as a limit in concentration but the averaging influence of a large population. They obtain convergence and fluctuations results for the empirical measure of the population, in Wasserstein distance. At last, \cite{Bonnet} retrieves the convergence establish in \cite{CDMR}, using Poisson random measures. Her work is motivated by the modelling of blood cancer and this explains why her work deals with linear or rational rate functions. Moreover, she does not study stronger convergence results.

In this article, we study the process $Z^N=(X^N,Y^N)\in\R^n\times\N^d$ defined as 
\begin{equation*}
\left\{\begin{aligned}
X^N(t) = & X^N_0 + \sum_{r\in\RR_C}\frac{h_r}{N}\int_{[0,t]\times\R_+}\ind_{u\leq N\lambda_r(Z^N(s^-))}\,\QQ_r(dsdu)\\
  & + \sum_{r\in\RR_D}\frac{h_r}{N}\int_{[0,t]\times\R_+}\ind_{u\leq \mu_r(Z^N(s^-))}\,\QQ_r(dsdu)\\
Y^N(t) = & Y^N_0 + \sum_{r\in\RR_D} e_r\int_{[0,t]\times\R_+}\ind_{u\leq \mu_r(Z^N(s^-))}\,\QQ_r(dsdu)\\
\end{aligned}\right.
\end{equation*}
where the $(\QQ_r)_{r\in\RR}$ are independent Poisson random measures. Notations and assumptions will be clarified in the next section. The two main results of our article are the convergence in probability of $(Z^N)_{N\geq1}$, in Theorem \ref{prop:LGN-s} and a central limit theorem, in Theorem \ref{prop:TCL-s}. As we stressed out, the convergence in probability for the uniform topology, has been proved for some restricted models (see \cite{PS}) without explicit speed of convergence. By contrast, we even reach $\L^1$ convergence, under some additional boundedness assumptions, and the speed of convergence. A convergence result is not really complete, nor usable without a central limit theorem for the fluctuations. To the best of our knowledge, a central limit theorem for hybrid jump process is a novelty of our work and has not been studied in the literature yet.

Let us describe the structure of this article. In Section \ref{sec:model}, we present our mathematical model and the standard assumptions under which it is well-posed. In Section \ref{sec:lgn}, we prove the convergence, in probability, for the uniform topology, of the jump process to a piecewise deterministic process. We also obtain a speed of convergence which suggests how to study the fluctuations of the continuous scale. Section \ref{sec:tcl} is dedicated to the central limit theorem. We show that the fluctuations process converges to the solution of a stochastic differential equation (SDE). In this case, the convergence is weaker in the probabilistic sense, and requires more complex tools.

\section{Gene network stochastic model}
\label{sec:model}
We consider chemical species, indexed by $i$ from $1$ to $n+d$, subject to a finite set of chemical reaction $R_r$, $r\in\RR$. Let $\ZZ\in\N^{n+d}$ be the vector consisting of the number of each species. Each reaction $R_r$ induces a transformation of system $\ZZ\to \ZZ+\gamma_r$, with rate $\Lambda_r(\ZZ)$. It results that, the process $Z$ evolves as a Markov process whose law is completely described by its generator
\[\LL f(z) =\sum_{r\in\RR}\left[f(z+\gamma_r)-f(z)\right]\Lambda_r(z).\]
We model a multiscale system with macroscopic and microscopic quantities of species. Let $N$ be a scaling parameter. The vector $\ZZ$ admits the decomposition $\ZZ=(N X^N,Y^N)\in\N^n\times\N^d$. The first component, describes the species in large abundance and is proportional to the scaling parameter. The vector $X^N$ can be interpreted as  a vector of concentrations. This is the concentration, or continuous, scale. The second component describes the species in scarce quantity, and so, is not rescaled. This is the discrete scale.
The set of reactions is also partitioned in two classes $\RR = \RR_C\cup\RR_D$, according to the species involved. For $r\in\RR_C$, the reaction only involves abundant species and has a rate proportional to the scaling parameter, i.e.  
\[\gamma_r =(h_r,0)\in\Z^n\times\Z^d,\quad \Lambda_r(\ZZ) = N \lambda_r(X^N,Y^N),\quad\forall r\in\RR_C.\]
This can be interpreted as quick reactions on the concentration scale. One the other hand, for $r\in\RR_D$, the reactions involve species from both scales and are slow 
\[\gamma_r =(h_r,e_r)\in\Z^n\times\Z^d,\quad \Lambda_r(\ZZ) =  \mu_r(X^N,Y^N),\quad\forall r\in\RR_D.\]
The law of the process $Z^N =(X^N,Y^N)$ is characterised by its generator
\begin{align*}
\LL_N f(x,y) =& \sum_{r\in\RR_C}\left[f\left(x+\frac{h_r}{N}, y\right)-f(x,y)\right]N\lambda_r(x,y)\\ 
&+ \sum_{r\in\RR_D}\left[f\left(x+\frac{h_r}{N}, y+ e_r\right)-f(x,y)\right]\mu_r(x,y).
\end{align*}
As explained in \cite{CDMR}, the sequence $(Z^N)_N$ converges in distribution to a PDMP whose law is characterised by its generator
\[\LL_\infty f(x,y) = F(x,y)\cdot\nabla_x f(x,y) + \sum_{r\in\RR_D} \left[f(x,y+e_r)-f(x,y)\right]\mu_r(x,y),\]
where $F : \R^n\times\N^d\to\R^n$ is defined as
\[F(x,y) = \sum_{r\in\RR_C} h_r\lambda_r(x,y).\]
However, there are several ways to construct a sequence of processes $(Z^N)_N$ and a process $Z=(X,Y)$, with these prescribed laws, and all of them may not allow anything stronger than the convergence in distribution. The classical construction, from \cite{CDMR} or its spatial generalisation \cite{DNN21}, uses independent Poisson processes, indexed by random clocks depending on the rates $\Lambda_r$. This coupling is not appropriate as the jump times of $Y^N$ and $Y$ are different with large probability. Hence, $\sup|Y^N-Y|$ does not converge to $0$. In this work, we use a coupling by random Poisson measures. As explained in \cite{HPV}, this new coupling is particularly adapted to $\L^1$ convergence for the uniform topology of discrete processes.

Let $\QQ_r$, $r\in\RR$ be independent Poisson random measures on $\R_+^2$ with intensity $dsdu$ (see \cite{Sato} for definition and properties of Poisson random measures). Let $Z^N_0=(X_0^N, Y_0^N)_{N\geq1}\in\N^n\times\N^d$ be a sequence of independent integrable random variables, independent of $\QQ_r$, $r\in\RR$. We define the sequence of processes $Z^N=(X^N, Y^N)$, for $N\geq1$, for $t\geq0$, by
\begin{equation}\label{eq:XN}
\left\{\begin{aligned}
X^N(t) = & X^N_0 + \sum_{r\in\RR_C}\frac{h_r}{N}\int_{[0,t]\times\R_+}\ind_{u\leq N\lambda_r(X^N(s^-), Y^N(s^-))}\,\QQ_r(dsdu)\\
  & + \sum_{r\in\RR_D}\frac{h_r}{N}\int_{[0,t]\times\R_+}\ind_{u\leq \mu_r(X^N(s^-), Y^N(s^-))}\,\QQ_r(dsdu)\\
Y^N(t) = & Y^N_0 + \sum_{r\in\RR_D} e_r\int_{[0,t]\times\R_+}\ind_{u\leq \mu_r(X^N(s^-), Y^N(s^-))}\,\QQ_r(dsdu)\\
\end{aligned}\right.
\end{equation}

With minimal regularity assumptions, this system has a unique solution, defined on a stochastic interval $[0,\tau^N[$. Its solution could be extended to $\R_+$ by adding of a cemetery point. In the following, we will make some assumptions on the total number of jump on $[0,t]$, $J^N_t$
in order to avoid the cemetery construction.  
We also define a PDMP, $Z=(X,Y)$, associated to the limit generator $\LL_\infty$. Let $Z_0=(X_0,Y_0)$ be a random variable, independent of the sequence $(Z^N_0)_N$ and $(\QQ_r)_{r\in\RR}$. We have
\begin{equation}\label{eq:X}
\left\{\begin{aligned}
X(t) =& X_0 +\int_0^tF(X(s), Y(s))\, ds\\
Y(t) = &Y_0 + \sum_{r\in\RR_D} e_r\int_{[0,t]\times\R_+}\ind_{u\leq \mu_r(X(s^-), Y(s^-))}\,\QQ_r(dsdu)\\
\end{aligned}\right.
\end{equation}

The general conditions under which a PDMP is well-posed are given in \cite{Davis}. It combines regularity assumptions on the rates, and a control over $J_t$, the number of jumps of $Y$ on $[0,t]$ 
\[J_t = \sum_{r\in\RR_D} \int_{[0,t]\times\R_+}\ind_{u\leq \mu_r(X(s^-), Y(s^-))}\,\QQ_r(dsdu).\] For our analysis, we use the following hypothesis.

\begin{ass}\label{ass:standard}
The rates $\lambda_r,\mu_r : \R_+^n\times\N^d\to\R_+$, for $r\in\RR$ are measurable and locally Lipschitz continuous on the first variable. Moreover, for all $(x,y)\in\R_+\times\N$, the vector field $F(\cdot, y)$ determines a unique flow $\phi(t,x,y)$, defined for all $t\geq 0$, solution to
\[\frac{d}{dt}\phi(t,x,y) = F\left(\phi(t,x,y), y\right).\]
The numbers of jumps, on every compact time interval, have a bounded moment
\[\forall N\geq1, \forall t\geq0,\quad \E\left[J^N_t\right]<\infty, \quad \E[J_t]<\infty.\]
\end{ass}

In many biological models, the rates are polynomial, so our assumption is relevant. This also explain why we do not assume boundedness on the rates. This assumption is sufficient to ensure the existence of the PDMP for all time.

\begin{prop}
Under Assumption \ref{ass:standard}, the systems \eqref{eq:XN} for all $N\geq1$, and \eqref{eq:X}, have a unique solution, defined on $\R_+$.
\end{prop}

As explained in \cite{CDMR}, Assumption \ref{ass:standard} may be easy to check on explicit models, but it is hard to formulate a sufficiently general criterion under which it holds. However, if the rates are bounded, the number of jumps is easily bounded. 

For all $z=(x,y)\in\R^n\times\Z^d$, we denote by $|x|$ and $|y|$, the Euclidean norm on $\R^n$ and $\R^d$ respectively, and we endow the product space $\R^n\times\Z^d$ with the norm $|z| = |x|+|y|$. For $T>0$, we denote by $\DD([0,T])$ the Skorokhod space on $[0,T]$, consisting of functions, defined on $[0,T]$, $\R^{n+d}$-valued, right-continuous and left-limited. This space can be endowed by two different topologies. The first one is the uniform topology, associated to the distance $\sup_{0\leq s\leq T} |f(s) - g(s)|$ for all $f,g\in\DD([0,T])$. This topology is particularly adequate on the subspace $\CC([0,T])$ of continuous functions. On $\DD([0,T])$, the uniform topology gives a very strict version of convergence. That is why, we usually use a weaker topology, more adapted to jump processes : the Skorokhod topology. This topology is also metrisable. In the following, we use the distance
\[d_{\DD([0,T])}(f, g) = \inf_\lambda \max\left\{\esssup_{0\leq s\leq T}|\lambda'(s)-1|, \sup_{0\leq s\leq T}|f\circ\lambda(s)-g(s)|\right\},\]
where the infimum is taken over the set of continuous, strictly increasing function $\lambda : [0,T]\to[0,T]$ such that $\lambda(0)=0$ and $\lambda(T) = T$. Let us remark that this metric is not the usual Skorokhod metric, as $\left(\DD([0,T]), d \right)$ is not complete, but it defines the same topology as the usual Skorokhod metric, for which, the space is complete (see \cite{Bill} for more details on the Skorokhod space and \cite{Kern} for more intuition on the different topologies on $\DD$).

\section{Strong convergence}
\label{sec:lgn}
The goal of this section is to prove the strong convergence of $\left(Z^N\right)_{N\geq1}$ to $Z$, for the uniform topology on $\DD([0,T])$. To that end, we prove a $\L^1$ convergence under strengthened assumptions, and then, we recover the convergence result, under the standard regularity assumptions, thanks to a truncation argument, in the spirit of \cite{CDMR}. In this section, without explicit mention, we assume that we have the following bounds.

\begin{ass}\label{ass:EZbound}
For all $r\in\RR$, the rates $\lambda_r$ and $\mu_r$ are bounded and globally Lipschitz continuous. We denote by $L>0$ a common upper bound of the rates and their Lipschitz constants.
\end{ass}
Let us note that under this assumption, $F$ is globally Lipschitz continuous, with constant
\[C_C = L\sum_{r\in\RR_C}|h_r|.\] 
The global boundedness from Assumption \ref{ass:EZbound}, allows  obtaining Gr\"onwall's bounds, in a $\L^1$ sense, for each scale, and so a convergence result. We begin with the discrete scale. The following lemma relies on the random measures coupling between $Y^N$ and $Y$.

\begin{lemme}
There exists $C_D>0$ such that for all $T\geq0$
\[\E\left[\sup_{0\leq t\leq T} |Y^N(t) - Y(t)|\right] \leq \E[|Y^N_0-Y_0|] + C_D\int_0^T\E\left[\sup_{0\leq t\leq s} |Z^N(t) - Z(t)|\right]\, ds.\]
\end{lemme}
\begin{proof}
By definition, for all $0\leq t\leq T$, we have
\begin{align*}
|Y^N(t) - Y(t)| 
&\leq |Y^N(0) - Y(0)|  + \sum_{r\in\RR_D}|e_r|\int_{[0,t]\times\R_+} \left|\ind_{u\leq\mu_r(Z^N(s^-))}-\ind_{u\leq\mu_r(Z(s^-))}\right|\,\QQ_r(dsdu)\\
&\leq |Y^N(0) - Y(0)|  + \sum_{r\in\RR_D}|e_r|\int_{[0,T]\times\R_+} \ind_{u\in \left[\mu_r(Z^N(s^-)), \mu_r(Z(s^-))\right]}\,\QQ_r(dsdu)\\
\end{align*}

Hence, 
\begin{align*}
\E\left[\sup_{0\leq t\leq T} |Y^N(t) - Y(t)|\right] 
&\leq \E[|Y^N_0-Y_0|]  + \sum_{r\in\RR_D}|e_r|\E\left[\int_0^T \left|\mu_r(Z^N(s^-)) - \mu_r(Z(s^-))\right|\, ds\right]\\
&\leq \E[|Y^N_0-Y_0|]  + L\sum_{r\in\RR_D}|e_r|\int_0^T\E\left[\left|Z^N(s) - Z(s)\right|\right]\, ds\\
&\leq \E[|Y^N_0-Y_0|]  + C_D\int_0^T\E\left[\sup_{0\leq t\leq s} \left|Z^N(t) - Z(t)\right|\right]\, ds\\
\end{align*}
Hence, the result.
\end{proof}

In order to treat the continuous scale, we decompose the process ass the sum of a martingale, a   process with finite variation, and a reminder. For all $r\in\RR$, we denote by $\tilde{\QQ}_r$ the compensated measure, defined as
\[\tilde{\QQ}_r(dsdu) = \QQ_r(dsdsu) - dsdu.\]
For all $t\geq 0$, we have
\begin{align*}
X^N(t)-X(t)
= &X^N_0-X_0 + \sum_{r\in\RR_C}\frac{h_i}{N}\int_{[0,t]\times\R_+}\ind_{u\leq N\lambda_r(Z^N(s^-))}\,\tilde{\QQ}_r(dsdu)\\
 +& \sum_{r\in\RR_D}\frac{h_r}{N}\int_{[0,t]\times\R_+}\ind_{u\leq \mu_r(Z^N(s^-))}\,\QQ_r(dsdu)\\
 +& \int_0^t F\left(Z^N(s)\right) - F\left(Z(s)\right)\, ds
\end{align*}
Let us denote
\[M^N_t = \sum_{r\in\RR_C}\frac{h_i}{N}\int_{[0,t]\times\R_+}\ind_{u\leq N\lambda_r(Z^N(s^-))}\,\tilde{\QQ}_r(dsdu)\]
and 
\[\gamma^N_t = \sum_{r\in\RR_D}\frac{h_r}{N}\int_{[0,t]\times\R_+}\ind_{u\leq \mu_r(Z^N(s^-))}\,\QQ_r(dsdu).\]
We show that these terms both converge to $0$, with explicit rate. For the first term,  the proof relies on  martingale arguments. We prove a convergence in $\L^1$, at rate $\sqrt{N}$.

\begin{lemme}\label{prop:Mart}
For all $T>0$ there exists $C_1>0$ such that
\[\E\left[\sup_{0\leq t\leq T}\left|M^N_t\right|\right] \leq \frac{C_1}{\sqrt{N}}.\]
\end{lemme}
\begin{proof}
For all $N\geq 1$, the process $M^N$ is a martingale, as 	a combination of stochastic integral. For all  $T>0$, we have
\[\left|M^N_T\right|^2\leq \frac{\sharp\RR_C}{N^2}\sum_{r\in\RR_C}|h_r|^2\left(\int_{[0,T]\times\R_+}\ind_{u\leq N\lambda_r(Z^N(s^-))}\,\tilde{\QQ}_r(dsdu)\right)^2.\]
Then, using  Ito's isometry formula, for all $r\in\RR_C$, we have
\begin{align*}
\E\left[\left(\int_{[0,T]\times\R_+}\ind_{u\leq N\lambda_r(Z^N(s^-))}\,\tilde{\QQ}_r(dsdu)\right)^2\right] 
=& \E\left[\int_{[0,T]\times\R_+}\ind^2_{u\leq N\lambda_r(Z^N(s^-))}\, dsdu\right]\\
=& N \E\left[\int_0^T \lambda_r(Z^N(s))\, ds\right]\\
\leq & LNT\\
\end{align*}
Therefore, we have
\[\E\left[\left|M^N_T\right|^2\right]\leq \frac{\sharp\RR_CLT\sum_{r\in\RR_C}|h_r|^2}{N}.\]
Moreover, Doob's inequality implies that
\[\E\left[\sup_{0\leq t\leq T} \left|M^N_t\right|\right] \leq 2\E\left[\left|M^N_T\right|^2\right]^{1/2}.\]
This ends the proof.
\end{proof}

The convergence of $(\gamma^N)_{N\geq 1}$ is more straightforward. 
\begin{lemme}\label{prop:eps}
For all $T\geq0$ there exists $C_2>0$ such that
\[\E\left[\sup_{s\in[0,T]}\left|\gamma^N_s\right|\right] \leq \frac{C_2}{N}.\]
\end{lemme}
\begin{proof}
For all $N\geq1$ and $0\leq t\leq T$, we have almost surely
\[\left|\gamma^N_t\right|\leq \frac{1}{N}\sum_{r\in\RR_D}|h_r|\int_{[0,t]\times\R_+}\ind_{u\leq L}\, \QQ_r(dsdu).\]
It follows that
\[\E\left[\sup_{0\leq t\leq T}\left|\gamma_t^N\right|\right]\leq \frac{LT}{N}\sum_{r\in\RR_D}|h_r|.\]
This concludes the proof.
\end{proof}

Now, we are able to state the convergence under the boundedness assumption \ref{ass:EZbound}.
\begin{theorem}\label{prop:LGN}
Assume that $(Z_0^N)_{N\geq 1}$ converges to $Z_0$ in $\L^1$, then, under Assumption \ref{ass:EZbound}, for all $T>0$, $(Z^N)_{\N\geq1}$ converges to $Z$ for the uniform topology on $\DD([0,T])$. 
Moreover, if $(\sqrt{N}\E\left[\left|Z^N_0-Z_0\right|\right])_{N\geq1}$ is bounded, then there exists $C(T)>0$ such that for all $N\geq 1$
\[\E\left[\sup_{0\leq t\leq T} \left|Z^N(t)-Z(t)\right|\right] \leq \frac{C(T)}{\sqrt{N}} .\]
\end{theorem}

\begin{proof}
For all $0\leq t\leq T$, we have
\begin{align*}
\left|X^N(t)-X(t)\right|
&\leq \left|X^N_0-X_0\right| + \left|M^N_t\right|+ \left|\gamma^N_t\right| + \int_0^t\left|F(X^N(s), Y^N(s))- F(X(s), Y(s))\right|\, ds\\
&\leq \left|X^N_0-X_0\right| + \left|M^N_t\right|+ \left|\gamma^N_t\right| + C_C\int_0^T\left|Z^N(s)-Z(s)\right|\, ds .\\
\end{align*}
By combining the previous lemmata, we have
\begin{align*}
\E\left[\sup_{0\leq t\leq T} |Z^N_t-Z_t|\right] 
\leq & \E[|Z^N_0-Z_0|] +\frac{C_1}{\sqrt{N}} + \frac{C_2}{N}\\
&+ (C_C+C_D)\int_0^T\E\left[\sup_{0\leq t\leq s} |Z^N_t-Z_t|\right]\, ds\\
\end{align*}
By Gr\"owall's lemma, it implies that 
\[\E\left[\sup_{0\leq t\leq T} |Z^N_t-Z_t|\right] \leq \left(\E[|Z^N_0-Z_0|] +\frac{C_1}{\sqrt{N}} + \frac{C_2}{N}\right)e^{(C_C+C_D)T}.\]
Hence, the result.
\end{proof}

Our first main result is the convergence of $(Z^N)_{N\geq1}$, under the natural assumptions.

\begin{theorem}\label{prop:LGN-s}
Assume that $(Z_0^N)_{N\geq1}$ converges to $Z_0$ in $\L^1$. Then, under Assumption \ref{ass:standard}, for all $T>0$, the sequence $(Z^N)_{N\geq1}$ converges in probability to $Z$, for the uniform topology on $\DD([0,T])$.
\end{theorem}

The proof leans on a rather classical truncation argument, well detailed in \cite{CDMR}. 
\begin{proof}
Let $\theta\in\CC^\infty(\R_+)$, supported on $[0,2]$ and such that $\theta(u)=1$ for all $u\in[0,1]$. For $k\geq1$ and $r\in\RR$, we define the truncated rates by
\[\lambda_r^k(x,y) = \theta(|(x,y)|/k)\lambda_r(x,y),\quad \mu_r^k(x,y) = \theta(|(x,y)|/k)\mu_r(x,y),\quad (x,y)\in\R^n\times\N^d.\]
We define the processes $Z^N_k$ and $Z_k$, solutions of the systems \eqref{eq:XN} and \eqref{eq:X}, with truncated rates. Under, Assumption \ref{ass:standard}, these processes fulfil Assumption \ref{ass:EZbound}. Hence,for all $k\geq1$, $(Z^N_k)_{N\geq1}$ converges in probability to $Z_k$ for the uniform topology on $\DD([0,T])$. Let us introduce the stopping times
\[\tau_k = \inf\left\{t\in[0,T], |Z_k(t)|\geq k\right\},\quad \tau_k^N = \inf\left\{t\in[0,T], |Z^N_k(t)|\geq k\right\},\]
with the convention $\inf\emptyset=T$. Then, on $[0,\tau_k[$ (respectively $[0,\tau^N_k[$), we have $Z_k(t) = Z(t)$ (respectively  $Z^N_k(t) = Z^N(t)$. Let us fix $0<\delta<T$ and $0<\varepsilon<1$. Let us remark that, on the event 
\[\left\{\tau_{k-1}>T-\delta\right\}\bigcap\left\{ \sup_{0\leq s\leq T-\delta} |Z^N_k(s)-Z_k(s)|<\varepsilon\right\}, \] 
we have \[\sup_{0\leq s\leq T-\delta} |Z^N_k(s)|<k.\]
It results that, on this event, $\tau^N_{k}>T-\delta$. Moreover, $\tau_{k}>\tau_{k-1}$. So on this event, we also have $\tau_{k}>T-\delta$, and the processes  $Z^N$ and $Z^N_k$ coincide with $Z^N_k$ and $Z_k$ respectively. Then, we have
\[\sup_{0\leq s\leq T-\delta} |Z^N(s)-Z(s)| <\varepsilon.\]
Therefore, we deduce that
\[\P\left(\sup_{0\leq s\leq T-\delta} |Z^N(s)-Z(s)|\geq\varepsilon\right)
\leq \P\left(\tau_{k-1}\leq T-\delta\right) + \P\left(\sup_{0\leq s\leq T-\delta} |Z^N_k(s)-Z_k(s)|\geq\varepsilon\right).\]

By Assumption \ref{ass:standard}, the process $Z$ can not explode on $[0,T-\delta]$, then the sequence $(\tau_k)_{k\geq 1}$ converges to $T$ almost surely and for $k$ large enough, the first term is arbitrary small. Using Theorem \ref{prop:LGN}, for a fixed $k$ and $N$ large enough, the second term is also arbitrary small. Thus, $(Z^N)_{N\geq1}$ converges in probability to $Z$, for the uniform topology on $\DD([0,T-\delta])$. As it is true for any $T$ and $\delta$, we have the result.
\end{proof}

Let us remark that with a cemetery point construction and without assumptions on the number of jumps $J^N_t$, we would have proved that for all $\varepsilon>0$
\[\lim_{N\to\infty} \P\left(T<\tau^N, \sup_{0\leq t\leq T}|Z^N(t)-Z(t)|<\varepsilon\right)=1.\]
To conclude with convergence results, let us remark that the convergence of $Y^N$ is much stronger: with large probability, the sequence $Y^N$ is stationary.
\begin{coro}\label{prop:yn}
Assume that $(Z_0^N)_{N\geq1}$ converges to $Z_0$ in $\L^1$. Under Assumption \ref{ass:standard}, for all $T>0$ and $\varepsilon>0$, there exists $N_0\geq 1$ such that for all $N\geq N_0$
\[\P\left(Y^N(s) = Y(s),\quad\forall 0\leq s\leq T\right) >1-\varepsilon.\]
\end{coro}

\begin{proof}
As the processes $(Y^N)_{N\geq1}$ are discrete, we have
\[\P\left(\sup_{0\leq s\leq t}|Y^N(s)-Y(s)|\leq 1/2\right) = \P\left(Y^N(s)=Y(s),\quad\forall 0\leq s\leq t\right).\]
According to Theorem \ref{prop:LGN-s}, we have 
\[\lim_{N\to+\infty}\P\left(\sup_{0\leq s\leq t}|Y^N(s)-Y(s)|>1/2\right) = 0.\]
This concludes the proof.
\end{proof}

\section{Central limit theorem}
\label{sec:tcl}
The goal of this section is to study the fluctuations of the continuous scale and to prove a central limit result. In the previous section, we have proved that, whenever the sequence $(\sqrt{N}Z^N_0)_{N\geq1}$ is bounded, the sequence $\left(\sqrt{N}\left(Z^N(t)-Z(t)\right)\right)_{N\geq1}$ is bounded in $\L^1$. This suggests a central limit behaviour, with fluctuations of order $\sqrt{N}$. However, Corollary \ref{prop:yn} suggests that it is not relevant to study the fluctuations of the discrete scale. Let us define
\[V^N(t) = \sqrt{N}\left(X^N(t)-X(t)\right),\quad V^N_0 =\sqrt{N}\left(X^N_0-X_0\right),\]
the fluctuations of the continuous scale. We show that $(V^N)_N$ converges in distribution, for a weaker topology, and we characterise its limit. Our proof is inspired by the classical mono-scale case, as in \cite{EthKur}, adapted to our Poisson random measure representation and to the multiscale contribution. In order to do so, we need more regularity on the rates. As for Section \ref{sec:lgn}, we begin with strong boundedness assumptions, which can be relaxed latter, thanks to a truncation argument. Here on, we make the following assumptions.

\begin{ass}\label{ass:EZPZbound}
For all $r\in\RR$, the rates $\lambda_r$ and $\mu_r$ are bounded and globally Lipschitz continuous. Moreover, for all $r\in\RR_C$, the rate $\lambda_r$ have $\CC^2$-regularity with respect to the first variable, with bounded derivatives up to order two. We denote by $L>0$ a common upper bound of the rates, their Lipschitz constants and their derivatives.
\end{ass}

We denote by $\nabla_x F$ and $\Hess_xF$ the first and second order differential of $F$ with respect to its $\R^n$ coordinates. Let us notice that for all $z\in\R_+^n\times \N^d$, $\nabla_xF(z)$ is a linear application $\R^n\to\R^n$ and $\Hess_x F(z)$ is bilinear $\R^n\times\R^n\to\R^n$.

As for the convergence result, we decompose $V^N$ so as to highlight the martingale part, the drift part and the different reminders. The idea is to make
a Taylor expansion of $F$ appear.  For all $N\geq1$ and $t\geq0$, we have
\begin{align*}
V^N(t)
=& V^N_0 + \sqrt{N}M^N_t + \sqrt{N}\gamma^N_t + \sqrt{N}\int_0^t F\left(Z^N(s)\right) - F\left(Z(s)\right)\, ds\\
=& V^N_0 + U^N_t + \sqrt{N}\gamma^N_t +\int_0^t \left\langle\nabla_xF\left(Z^N(s)\right), V^N(s)\right\rangle\, ds + \zeta^N_t + \xi^N_t\\
\end{align*}
where 
\begin{align*}
U^{N}_t &= \sum_{r\in\RR_C}\frac{h_r}{\sqrt{N}}\int_{[0,t]\times\R_+}\ind_{u\leq N\lambda_r(Z^N(s))}\,\tilde{\QQ}_r(dsdu)\\
\zeta^N_t &= \sqrt{N}\int_0^t F\left(X^N(s),Y^N(s)\right) - F\left(X(s),Y^N(s)\right)-\left\langle\nabla_x F(Z^N(s),\frac{V^N(s)}{\sqrt{N}}\right\rangle\, ds\\
\xi^N_t &= \sqrt{N}\int_0^1F\left(X(s),Y^N(s)\right) - F\left(X(s),Y(s)\right)\, ds\\
\end{align*}
 
Let us note that we need to distinguish the behaviour of $(X^N)_{N\geq1}$ and $(Y^N)_{N\geq1}$ as the latter is never close to its limit without being equal to it. These different terms need to be treated through different strategies. On the one hand, $\sqrt{N}\gamma^N$ and $\zeta^N$ converge in a strong sense. For the former, it is a direct application of Lemma \ref{prop:eps}.  It yields that $\sqrt{N}\gamma^N$ converges to $0$ in $\L^1$ for the uniform topology. The sequence $(\zeta^N)_N$ only converges in probability, for the uniform topology.
\begin{lemme}
Assume that $(V^N_0)_N$ is bounded in $\L^1$, then $(\zeta^N)_N$ converges to $0$, in probability, for the uniform topology on $\DD([0,T])$. Moreover, there exists $C(T)>0$ such that for all $\varepsilon>0$
\[\P\left(\sup_{0\leq s\leq T}\left|\zeta^N_s\right|>\varepsilon\right) \leq \frac{C(T)}{\sqrt{\varepsilon} N^{1/4}}.\]
\end{lemme}
\begin{proof}
From Theorem \ref{prop:LGN}, the sequence $(V^N)_N$ is bounded in $\L^1$ for the uniform topology on $\DD([0,T])$ : there exists $c>0$ such that
\[ \E\left[\sup_{0\leq s\leq t} \left|V^N(s)\right|\right]\leq c.\]
On the other hand, from Taylor formula, for all $0\leq t\leq T$, we have
\begin{align*}
\zeta^N_t
=& \sqrt{N}\int_0^t F\left(X^N(s),Y^N(s)\right) - F\left(X(s),Y^N(s)\right)-\left\langle\nabla_x F(X^N(s),Y^N(s)),\frac{V^N(s)}{\sqrt{N}}\right\rangle\, ds\\
=& \sqrt{N}\int_0^t\int_0^1 (1-u)\Hess_x F\left(X^N_s -u\frac{V^N(s)}{\sqrt{N}}, Y^N_s\right)\left(\frac{V^N(s)}{\sqrt{N}},\frac{V^N(s)}{\sqrt{N}}\right)\, du\, ds\\
\end{align*}
Thus, there exists $C>0$, such that, for all $0\leq t\leq T$
\[\left|\zeta^N_t\right|\leq \frac{C}{\sqrt{N}}\sup_{0\leq s\leq T}|V^N(s)|^2.\]
Hence, from Markov's inequality, for all $\alpha>0$, we have
\[\P\left(\sup_{0\leq t\leq T}\left|\zeta^N_t\right|>\varepsilon\right)\leq \frac{c\sqrt{C}}{\sqrt{\varepsilon}N^{1/4}}\underset{N\to+\infty}{\longrightarrow} 0.\]
This ends the proof.
\end{proof}

On the other hand, the convergence of the processes $U^N$ and $\xi^N$ is, essentially, a convergences in distribution only. The key argument, to obtain this convergence, is the tightness. In the case of $\xi^N$, it results from Arzel\`a-Ascoli theorem directly. This term is specific to our multiscale model. It describes the contribution of the discrete scale on $V^N$. As we remarked, the sequence $(Y^N)_N$ is stationary with large probability. This suggests that $\xi^N$ does not have any contribution. This allows lifting the convergence in distribution to a convergence in probability.

\begin{lemme}
Assume that $(V_0^N)_N$ is bounded, then, for all $T>0$ the sequence $\xi^N$ converges in probability, for the uniform topology on $\DD([0,T])$, to $0$.
\end{lemme} 
 
\begin{proof}
Firstly, the sequence $(\xi^N)_{N\geq 1}$ is tight in $\CC([0,T])$. Indeed, for all $N$, $\xi^N$ is differentiable and we have

\begin{align*}
\|\xi^N\|_{Lip,[0,T]}
:=& \sup_{0\leq s\leq T}\left|\xi^N_s\right| + \sup_{0\leq s\leq T}\left|{\xi^N_s}'\right|\\
\leq & \sqrt{N}\int_0^T|F(X(s), Y^N(s)) -F(X(s), Y(s))|\, ds\\ 
& + \sqrt{N}\sup_{0\leq s\leq T}|F(X(s), Y^N(s)) -F(X(s), Y(s))|\\
\leq & (1+T)\sqrt{N}\sup_{0\leq s\leq T}|F(X(s), Y^N(s)) -F(X(s), Y(s))|\\
\leq & (1+T)C_C\sqrt{N}\sup_{0\leq s\leq T}|Z^N(s)-Z(s)|
\end{align*}

Besides, from Theorem \ref{prop:LGN}, $(\sqrt{N}(Z_N-Z))_{N\geq1}$ is bounded in $\L^1$ for the uniform topology on $\DD([0,T])$. So, there exist a constant $C>0$ such that
\[\E\left[\|\xi^N\|_{Lip,[0,T]}\right]\leq C.\]
For all $\varepsilon>0$, we define $K_\varepsilon$ the $\|\cdot\|_{Lip,[0,T]}$-ball of $\CC([0,T])$, of radius $C/\varepsilon$.  Hence, for all $\varepsilon>0$, we have 
\[\P\left(\xi^N\in K_\varepsilon\right)\geq 1-\varepsilon.\]
From Arzel\`a-Ascoli theorem, $K_\varepsilon$ is a compact of $\CC([0,T])$, endowed with the uniform topology. Therefore, $(\xi^N)_N$ is $\CC$-tight and admits limit points for the uniform convergence. Let $\xi$ be such a limit point. According to Skorokhod's representation theorem on $\CC([0,T])$ (see \cite[Theorem 6.7]{Bill}), up to extraction, we can assume that $(\xi^N)_N$ and $(Y^N)_N$ converge to $\xi$ and $Y$ almost surely, for the uniform topology. Hence, according to Corollary \ref{prop:yn}, there exists $N_0$ such that for all $N\geq N_0$, $\sup_{[0,t]}|Y^N(s)-Y(s)|=0$ and $\xi^N=0$. Then, $(\xi^N)_N$ admits a unique limit point, $0$, and converges in distribution, and so in probability, to this unique limit point, for the uniform topology. 
\end{proof}

Finally, the tightness of $(U^N)_{N\geq 1}$ and $(V^N){N\geq 1}$ requires a more involved argument, Aldous' criterion, which ensure the convergence for the Skorokhod topology only.

\begin{lemme}{Aldous' criterion \cite[Theorem VI$4.5$]{JacShi}}
Let $(H^N)_N$ be a sequence of càdlàg adapted processes, defined on $[0,t]$ and $\R^n$-valued. Assume that it satisfies
\begin{enumerate}
\item[(i)] For all $\varepsilon>0$, there are $N_0\in\N^*$ and $K\in\R$ with
\[N\geq N_0 \Rightarrow \P\left(\sup_{0\leq s\leq t}|H^N_s|>K\right) \leq \varepsilon.\]
\item[(ii)] For all $\varepsilon,\eta>0$, there exists $\delta_0>0$ and $N_0\in\N^*$ such that for all $N\geq N_0$, for all stopping times $\delta$ and $\tau$ such that $\tau+\delta\leq t$ and $\delta\leq \delta_0$ a.s
\[\P\left(|H^N_{\tau+\delta} - H^N_\tau|\geq\varepsilon\right)\leq\eta.\] 
\end{enumerate}
Then the sequence $(H^N)_N$ is tight for the Skorokhod topology.
\end{lemme}

\begin{lemme}
For all $T>0$, $(U^N)_{N\geq 1}$ converges in distribution for the uniform topology on $\DD([0,T])$. Moreover, its limit is the law of a process
\[U_t=\int_0^t\sigma_s\, dB_s,\] 
where $B$ is a $\sharp\RR_C$-dimensional Brownian motions and $\sigma\in\MM_{n,\sharp\RR_C}$ satisfies
\[(\sigma_t)_{ir} = h^i_r\lambda_r^{1/2}(Z(t)).\]  
\end{lemme}

\begin{proof}
This proof is divided in three steps : tightness of the sequence, through Aldous' criterion, continuity of limit points and uniform convergence, and uniqueness. 

\textbf{Tightness.} From Lemma \ref{prop:Mart}, there exists $C>0$ such that for all $N\geq1$
\[ \E\left[\sup_{0\leq s\leq T} |U^{N}_s|\right]\leq C.\]
Then, from Markov's inequality, $(U^{N})_N$ satisfies the condition $(i)$. Let $\tau$ and $\delta$ be two finite stopping times. We have
\[\E\left[\left|U^{N}_\tau - U^{N}_{\tau+\delta}\right|^2\right]
= \sharp\RR_C \sum_{r\in\RR_C}|h_r|^2\E\left[\int_\tau^{\tau+\delta}\lambda_r(X^N(s), Y^N(s))\, ds\right]\leq C_C^2/L \E[\delta].\]
Then for all $\varepsilon>0$,  $m>0$ and $\eta>0$, we set $\delta_0=\eta\varepsilon^2 L/C_C^2$. For all stopping times $\tau\leq m$ and $\delta\leq \delta_0$,
\[\P\left(|U^{N}_\tau-U^{N}_{\tau+\delta}|\geq\varepsilon\right)\leq\eta.\]
Thus, according to Aldous' criterion, the sequence $(U^{N})_{N\geq1}$ is tight. Thanks to Prokhorov's theorem, the sequence $(U^{N})_{N\geq1}$ admits limit points.

\textbf{Continuity.} The jumps of $U^N$ are uniformly bounded and their sizes converge to $0$. Indeed, for all $N\geq1$ and all $0\leq s\leq t$, we have 
\[|\Delta U^{N}_s|\leq\frac{\sup_{r\in\RR_C}|h_r|}{\sqrt{N}}.\] 
Thus, for all $\varepsilon>0$, we have
\[\lim_{N\to\infty}\P\left(\sup_{0\leq s\leq T} |\Delta U^N_s| \geq\varepsilon\right)=0.\]
According to the $\CC$-tightness criterion \cite[Proposition VI$3.26$]{JacShi}, any limit distribution of $(U^{N})_N$ is the law of a continuous process. As $[0,T]$ is compact, limit points of $(U^N)_{N\geq1}$ are uniformly continuous and the convergence for the Skorokhod topology implies for the uniform convergence.
 
\textbf{Uniqueness.} Let $U$ be a limit point of $(U^{N})_{N\geq1}$ for the uniform convergence. Up to taking a subsequence, let us assume that $(U^{N})_{N\geq1}$ converges in distribution to $U$. As $(U^{N})_{N\geq1}$ is a sequence of martingale, then  $U$ is a continuous martingale. Moreover, as the jumps sizes converge to $0$, the convergence of $(U^N)_{N\geq1}$ implies the joint convergence, in distribution, of its quadratic variation $([(U^{N})^i,(U^{N})^j])_{1\leq i,j\leq n}$ to the quadratic variation of $U$. (see \cite[Corollary VI.6.6]{JacShi}). Let us denote the $\R$-valued processes, for $N\geq1$ and $r\in\RR_C$
\[H^{r,N}_t = \int_{[0,t]\times\R_+}\ind_{u\leq N\lambda_r(Z^N(s^-))}\,\tilde{\QQ}_r(dsdu).\]
Let us remark that for all $r\neq\rho\in\RR_C$, as $\tilde{\QQ}_r$ and $\tilde{\QQ}_\rho$ are independent, $[ H^{r,N},H^{\rho,N}]_t = 0$.
Then, for all $1\leq i,j\leq n$, the quadratic variations $[(U^{N})^i,(U^{N})^j]_t$ write as
\begin{align*}
[(U^{N})^i,(U^{N})^j]_t
=& \sum_{r,\rho\in\RR_C} \frac{h^i_r h^j_\rho}{N} [H^{r,N},H^{\rho,N}]_t\\
=& \sum_{r\in\RR_C} \frac{h^i_r h^j_r}{N} [H^{r,N},H^{r,N}]_t\\
\end{align*}
For each $N\geq1$ and $r\in\RR_C$, $H^{r,N}$ are compensated Poisson processes, and we have
\[[ H^{r,N},H^{r,N}]_t = \sum_{s\leq t}(\Delta H^{r,N}_s)^2= \sum_{s\leq t}\Delta H^{r,N}_s = H^{r,N}_t + N\int_0^t\lambda_r(Z^N(s))\, ds.\]
Then, from the convergence of $(Z^N)_{N\geq 1}$, and from the boundedness assumption on $\lambda_r$, we have the following convergences, in probability, for the uniform topology
\[\lim_{N\to\infty}\frac{H^{r,N}}{N} = 0,\quad\text{and}\quad \lim_{N\to\infty} \int_0^t\lambda_r(Z^N(s))\, ds = \int_0^t\lambda_r(Z(s))\, ds.\]
It follows that the quadratic variation of $U$ satisfies
\[[ U^i,U^j]_t = \sum_{r\in\RR_C} h^i_r h^j_r \int_0^t\lambda_r(Z(s))\, ds.\]
Therefore, the unique limit point of $(U^N)_{N\geq1}$, for the convergence in distribution, for the uniform topology, is the Gaussian martingale with quadratic variation  
\[[ U^i,U^j]_t = \sum_{r\in\RR_C} h^i_r h^j_r \int_0^t\lambda_r(Z(s))\, ds,\]
and so, $(U^N)_{N\geq1}$ converges to this martingale, for the uniform topology on $\DD([0,T)]$.
To conclude, we have the decomposition
\[d[ U^i,U^j]_t= \sum_{r\in\RR_C} h^i_r h^j_r \lambda_r(Z(t)) dt = (\sigma_t \sigma^*_t)_{ij}dt,\] 
with $(\sigma_t)_{ir} = h^i_r\lambda_r^{1/2}(Z(t))$ for $1\leq i\leq n$ and $r\in\RR_C$.
Then, using a martingale representation theorem (see \cite[Theorem II.7.1]{IW} for example), up to an enlargement of the probability space, there exists a $\sharp\RR_C$-dimensional Brownian motion $B$, such that
\[U_t = \int_0^t\sigma_s\, dB_s.\]
\end{proof}

Let us note that $U$ can also be represented as a combination of independent Brownian motions with random clocks 
\[U_t = \sum_{r\in\RR_C} h_r B^r_{\theta_r(t)},\quad \theta_r(t) = \int_0^t\lambda_r(Z(s))\, ds.\]
Now, under the boundedness assumption \ref{ass:EZPZbound}, we can prove the convergence of the fluctuation process  $V^N$.

\begin{theorem}\label{prop:TCL}
Let us assume that $(V^N_0)_{N\geq 1}$ is bounded in $\L^1$ and converges in distribution to $V_0$. Then, under Assumptions \ref{ass:EZPZbound}, $(V^N)_{N\geq 1}$ converges in distribution, for the uniform topology on $\DD([0,T])$, to the law of $V$, the solution to the SDE
\begin{equation}\label{eq:V}
dV(t) = \sigma_t dB_t +\left\langle\nabla_xF\left(Z(t)\right), V(t)\right\rangle\, dt.
\end{equation}
\end{theorem}

\begin{proof}
Firstly we show that the sequence $(V^N)_N$ is tight. Indeed, from Theorem \ref{prop:LGN}, $(V^N)_N$ is bounded in $\L^1$ and so, from Markov's inequality, it satisfies the boundedness condition $(i)$ of Aldous' criterion.  Moreover, let $\tau$ and $\delta$ be two stopping times such that $\tau+\delta\leq t$, we have
\begin{align*}
|V^N({\tau+\delta})-V^N(\tau)|
\leq & |U^N_{\tau+\delta} -U^N_\tau|+2\sqrt{N}\sup_{0\leq s\leq t}|\gamma^N_s| +\sup_{0\leq s\leq t}|\zeta^N_s| + C_C\int_\tau^{\tau+\delta}|V^N(s)|\, ds\\
 &+ \int_\tau^{\tau+\delta} \left|F\left(X(s), Y^N(s)\right)-F\left(X(s), Y(s)\right)\right|\, ds\\
\leq & |U^N_{\tau+\delta} - U^N_\tau|+2\sqrt{N}\sup_{0\leq s\leq t}|\gamma^N_s| + \sup_{0\leq s\leq t}|\zeta^N_s|+ C_C\delta\sup_{0\leq s\leq t}|V^N(s)|\\
& +2C_C\delta\\
\end{align*}
Then, for all $\alpha>0$, we have
\begin{align*}
\P\left(|V^N({\tau+\delta})-V^N(\tau)|\geq\varepsilon\right)
\leq& \frac{\E\left[ |U^N_{\tau+\delta} -U^N_\tau|^2\right]}{\varepsilon^2} + 2\frac{\E\left[\sqrt{N}\sup_{0\leq s\leq t}|\gamma^N_s|\right]}{\varepsilon}\\
& + \P\left(\sup_{0\leq s\leq t}|\zeta^N_s|\geq\varepsilon\right) +C_C\frac{\|\delta\|_{\L^\infty}\E\left[\sup_{0\leq s\leq t}|V^N(s)|\right]}{\varepsilon}\\
& + 2C_c\P(\delta\geq\varepsilon)\\
\leq& \frac{L_C\|\delta\|_{\L^\infty}}{\varepsilon^2} + \frac{2C}{\sqrt{N}\varepsilon} + \frac{C}{\sqrt{\varepsilon}N^{1/4}} + \frac{C_C\|\delta\|_{\L^\infty}}{\varepsilon} +\frac{2C_C \|\delta\|_{\L^\infty}}{\varepsilon}\\
\leq& \frac{C\|\delta\|_{\L^\infty}}{\varepsilon} + \frac{C}{\sqrt{\varepsilon}N^{1/4}}\\
\end{align*}
for some $C>0$. Then, for all $\varepsilon,\eta>0$, there exist $\delta_0 = \frac{\eta\varepsilon}{2C}$ and $N_0 = \frac{16C^4}{\varepsilon^2\eta^4}$ such that for all $N\geq N_0$, for all stopping time $\tau, \delta$ with $\delta\leq \delta_0$ a.s and $\tau+\delta\leq t$ a.s
\[\P\left(|V^N({\tau+\delta})-V^N(\tau)|\geq\varepsilon\right)\leq \eta.\]
According to Aldous' criterion, the sequence $(V^N)_{N\geq 1}$ is tight. This means that it admits limit points and up to extraction, we can assume that $(V^N)_{N\geq 1}$ converges almost surely for the Skorokhod topology. Using the previous lemmata, $(U^N)_{N\geq 1}$, $(\sqrt{N}\gamma^N)_{N\geq 1}$, $(\zeta^N)_{N\geq 1}$ and $(\xi^N)_{N\geq 1}$ converge for the uniform topology.  Hence, up to a new extraction, we can assume that 
\[\bar{U}^N = V^N_0 + U^N+\sqrt{N}\gamma^N+\zeta^N+\xi^N\]
converges almost surely to $\bar{U} = V_0 + U$, and  $(Z^N)_{N\geq 1}$ converges almost surely to $Z$, for the uniform topology. We show that $(V^N)_{N\geq 1}$ converges to $V$ almost surely, for the uniform topology. Indeed, for  all $t\leq T$, we have
\begin{align*}
V^N({t}) - V(t)
= & \bar{U}^N_{t}-\bar{U}_t + \int_0^{t} \left\langle\nabla_x F\left(Z^N(s)\right),V^N(s)-V(s) \right\rangle\, ds\\ 
& + \int_0^t\left\langle\nabla_x F\left(Z^N(s)\right) -\nabla_x F\left(Z(s)\right),V(s)\right\rangle\, ds\\
\sup_{0\leq\sigma\leq t} |V^N({\sigma}) - V(\sigma)|
\leq & \sup_{0\leq\sigma\leq T} |\bar{U}^N_{\sigma} - \bar{U}_\sigma| + C_C\int_0^t\sup_{0\leq\sigma\leq s} |V^N({\sigma}) - V(\sigma)|\, ds\\
& + C_CT\sup_{0\leq\sigma\leq T} |V(\sigma)|\sup_{0\leq\sigma\leq T} |Z^N({\sigma}) - Z(\sigma)|\\
\leq & C(T)\left(\sup_{0\leq\sigma\leq T} |\bar{U}^N_{\sigma} - \bar{U}_\sigma| + \sup_{0\leq\sigma\leq T} |V(\sigma)|\sup_{0\leq\sigma\leq T} |Z^N({\sigma}) - Z(\sigma)|\right)\\
\end{align*}
where the last inequality follows from Gr\"onwall's lemma, and $C>0$ is deterministic. Almost surely, $V$ is continuous on $[0,T]$ and so $\sup_{0\leq\sigma\leq T} |V_\sigma|$ is finite. Then, we have proved that any convergent subsequence of $(V^N)_{N\geq1}$, for the Skorokhod topology, converges to $V$ for the uniform topology. It follows that $(V^N)_{N\geq1}$ converges to $V$, for the uniform topology on $\DD\left([0,T]\right)$.
\end{proof}

Let us remark that this proof would have been more tedious is we were not able to obtain a convergence for the uniform topology, in the previous lemmata. Indeed, products and sums are not continuous for the Skorokhod topology. In particular, the drift term
\[\int_0^t\left\langle \nabla_xF\left(Z^N(s)\right), V^N(s)\right\rangle\, ds\]
would have required some careful control of the number of jumps on the discrete scale. 

To conclude, we apply the truncation argument to prove a CLT under Assumptions \ref{ass:standard}.

\begin{theorem}\label{prop:TCL-s}
Let us assume that for all $r\in\RR_C$, the rates $\lambda_r$ have $\CC^2$ regularity, and that $(V^N_0)_{N\geq 1}$ is bounded in $\L^1$ and converges in distribution to $V_0$. Then, under Assumptions \ref{ass:standard}, $(V^N)_N$ converges in distribution, for the uniform topology on $\DD([0,T])$, to $V$.
\end{theorem}

\begin{proof}
Firstly, let us remark that under Assumption \ref{ass:standard}, equation \eqref{eq:V} has a unique solution. Indeed, it is linear, with, almost surely, finite coefficients. Then, we introduce the truncated fluctuations $V^N_k = \sqrt{N}\left(X^N_k-X_k\right)$. For all $k\geq1$, the sequence $(V^N_k)_{N\geq1}$ satisfies the assumptions of Theorem \ref{prop:TCL}.

Then, we use the truncation argument to show that the sequence $(V^N)_{N\geq1}$ is tight on $\DD([0,T-\delta])$ and that it admits a unique limit point. Indeed, since $\P(\tau_{k}\leq T-\delta)$ converges to $0$, for all $\eta>0$ there exists $k\geq 1$ such that $\P(\tau_{k-1}\leq T-\delta)<\eta/3$. Besides, as $(Z^N_k)_N$ converges to $Z_k$, there exist $N_0\geq 1$ and $0<\varepsilon<1$ such that for all $N\geq N_0$ 
\[\P\left(\sup_{0\leq s\leq T-\delta}|Z^N_k(s)-Z_k(s)|>\varepsilon\right)<\eta/3.\]
Moreover, from Theorem \ref{prop:TCL}, the sequence $(V^N_k)_{N\geq1}$ satisfies Aldou's criterion. So there exist $N_1\geq 1$ and $M>0$ such that for all $N\geq N_1$
\[\P\left(\sup_{0\leq s\leq T-\delta}|V_k^N(s)|>M\right)<\eta/3.\] 

Then, under $\tau_{k}\leq T-\delta$, $\sup_{0\leq s\leq T-\delta}|Z_k^N(s)-Z_k(s)|\leq \varepsilon$ and $\sup_{0\leq s\leq T-\delta}|V_k^N(s)|\leq M$, we have $\tau^N_k>T-\delta$ and
\[\sup_{0\leq s\leq T-\delta}|V^N(s)|\leq M.\]
Hence, for all $N\geq N_0\vee N_1$, we have 
\begin{align*}
P\left(\sup_{0\leq s\leq T-\delta}|V^N(s)|>M\right)
\leq & \P\left(\tau_{k-1}\leq T-\delta\right) +\P\left(\sup_{0\leq s\leq T-\delta}|Z_k^N(s)-Z_k(s)|> \varepsilon\right)\\
&+ \P\left(\sup_{0\leq s\leq T-\delta}|V_k^N(s)|>M\right)\\
<& \eta
\end{align*}  
Thus, $(V^N)_{N\geq1}$ satisfies the first condition of Aldou's criterion. We can check the second condition with the same trick. Indeed, for all stopping time $\tau, \sigma$ such that $\tau+\sigma\leq T-\delta$, we have
\begin{align*}
\P\left(|V^N({\tau+\sigma}) - V^N(\tau)|\geq\varepsilon\right)
\leq & \P\left(\tau_{k-1}\leq T-\delta\right) +\P\left(\sup_{0\leq s\leq T-\delta}|Z_k^N(s)-Z_k(s)|> \varepsilon\right)\\
&+ \P\left(|V^N_k(\tau+\sigma) - V^N_k(\tau)|\geq\varepsilon\right)
\end{align*}
Therefore, the sequence $(V^N)_{N\geq1}$ is tight. Now, let us prove that it admits a unique limit point. Indeed, we have
\begin{align*}
\P\left(\sup_{0\leq s\leq T-\delta}|V^N(s)-V(s)|\geq\varepsilon\right)
\leq & \P\left(\tau_{k-1}\leq T-\delta\right) +\P\left(\sup_{0\leq s\leq T-\delta}|Z_k^N(s)-Z_k(s)|> \varepsilon\right)\\
&+ \P\left(\sup_{0\leq s\leq T-\delta}|V^N_k(s)-V_k(s)|\geq\varepsilon\right)
\end{align*}
Hence, whenever a subsequence of $(V^N)_{N\geq1}$ converges in distribution, up to another extraction, we can assume that $(V^N_k)_{N\geq1}$ converges in probability to $V_k$ and so the subsequence  of $(V^N)_{N\geq1}$ converges in probability to $V$. This proves that $V$ is the unique limit point  of $(V^N)_{N\geq1}$. Therefore,  $(V^N)_{N\geq1}$ converges in distribution, for the uniform topology on $\DD([0,T-\delta])$, to $V$. This ends the proof.
\end{proof}

\section*{Acknowledgement}
I would like to thank Arnaud Debussche for his suggestions and for our useful discussions. This research is partially supported by the Centre Henri Lebesgue (ANR-11-LABX-0020-0).

\bibliographystyle{abbrv}
\bibliography{biblio}
\end{document}